\documentclass[12pt]{article}

\usepackage{amsmath,amssymb,amsthm}

\usepackage{graphicx}

\def\e{{\rm e}}
\def\<{{\langle}}
\def\>{{\rangle}}
\def\R{{\mathbb R}}

\def\D{{\mathcal D}}
\def\N{{\mathbb N}}

\def\d{{\rm d}}
\def\ri{{\rm i}}
\def\Z{{\mathbb Z}}
\def\curl{{\rm curl}}
\def\cgw{\rightharpoonup}
\def\cgws{\stackrel{*}\rightharpoonup}
\def\eps{{\varepsilon}}

\def\l{\alpha}
\def\:{\colon }

\numberwithin{equation}{section}

\newtheorem{theorem}{Theorem}[section]

\newtheorem{lemma}[theorem]{Lemma}
\newtheorem{proposition}[theorem]{Proposition}

\newtheorem{definition}[theorem]{Definition}

\title{Using periodic boundary conditions to approximate the Navier--Stokes equations on $\R^3$ and the transfer of regularity}
\author{James C.\ Robinson\\
\small Mathematics Institute\\
\small University of Warwick\\
\small Coventry CV4 7AL. UK.}

\begin{document}

\maketitle

\begin{abstract}
  This paper considers solutions $u_\alpha$ of the three-dimensional Navier--Stokes equations on the periodic domains $Q_\alpha:=(-\alpha,\alpha)^3$ as the domain size $\alpha\to\infty$, and compares them to solutions of the same equations on the whole space. For compactly-supported initial data $u_\alpha^0\in H^1(Q_\alpha)$, an appropriate extension of $u_\alpha$ converges to a solution $u$ of the equations on $\R^3$, strongly in $L^r(0,T;H^1(\R^3))$, $r\in[1,\infty)$. The same also holds when $u_\alpha^0$ is the velocity corresponding to a fixed, compactly-supported vorticity.

  A consequence is that if an initial compactly-supported velocity $u_0\in H^1(\R^3)$ or an initial compactly-supported vorticity $\omega_0\in H^1(\R^3)$ gives rise to a smooth solution on $[0,T^*]$ for the equations posed on $\R^3$, a smooth solution will also exist on $[0,T^*]$ for the same initial data for the periodic problem posed on ${Q_\alpha}$ for $\alpha$ sufficiently large; this illustrates a `transfer of regularity' from the whole space to the periodic case.\end{abstract}

\section{Introduction}

The aim of this paper is to compare solutions of the Navier--Stokes equations
\begin{equation}\label{NSE}
\partial_tu-\Delta u+(u\cdot\nabla)u+\nabla p=0,\qquad\nabla\cdot u=0,
\end{equation}
posed on `large' periodic domains ${Q_\alpha}:=(-\alpha,\alpha)^3$ and on the whole space $\R^3$. One would expect, when the initial velocity is sufficiently localised, that the solutions on a `large enough' domain should mimic those on $\R^3$, and this approach is the basis of many numerical experiments. Indeed, discussions with Robert Kerr about his numerical investigations (Kerr, 2018) of the trefoil configurations of vorticity from the experiments of Scheeler et al.\ (2014) were the original motivation for this  paper, which gives a rigorous justification of this intuition. 

Section \ref{sec:vorticity} contains an analysis of the velocity fields that arise from such compactly-supported vorticities. The results there both provide a natural family of initial data to consider on the domains $Q_\alpha$, and also serve to illustrate of some of the arguments that follow in a relatively simple setting.

It is shown that given a fixed compactly-supported vorticity $\omega\in H^1(\R^3)$, the corresponding velocities $u_\alpha$ on $Q_\alpha$ have extensions to $\R^3$, $\tilde u_\alpha$, that converge strongly in $H^1(\R^3)$ to the velocity on $\R^3$ reconstructed from $\omega$ using the Biot--Savart Law. Obtaining strong convergence in $H^1(\R^3)$ requires uniform bounds on the `tails'
$$
\int_{x\in Q_\alpha:\ |x|\ge R}|\nabla u_\alpha|^2,
$$
a technique also employed later for solutions of the Navier--Stokes equations, and which goes back at least to Leray (1934).

After recalling some basic existence results for weak and strong solutions of the Navier--Stokes equations in Section \ref{sec:NSE}, it is shown that a subsequence of weak solutions on ${Q_\alpha}$ (solutions bounded in $L^2$ that satisfy the energy inequality) will converge to a weak solution on $\R^3$, given weak convergence of the initial data in $L^2(\R^3)$. This result goes back at least to Heywood (1988), who used it as a way of proving the existence of weak solutions on the whole space.

The main result of the paper concerns the convergence of strong solutions (i.e.\ solutions that remain bounded in $H^1$) given convergence of the initial data in $H^1(\R^3)$; due to uniqueness of the limiting solution this convergence now occurs without the need to extract a subsequence. By bounding the `tails' of $|u_\alpha|^2$ at infinity it is shown that $\tilde u_\alpha$ converges to $u$ strongly in $L^p(0,T;L^2(\R^3))$ for all $p\in[1,\infty)$, and then, via interpolation of the $H^1$ norm between $L^2$ and $H^2$, the boundedness of $u_\alpha$ in $L^2(0,T;H^2(\R^3))$ shows that $\tilde u_\alpha$ converges strongly to $u$ in $L^r(0,T;H^1(\R^3))$, $r\in[1,4)$.

Finally, using this strong convergence, comes what is perhaps the most striking result of the paper: if $u^0\in H^1(\R^3)$ with compact support (or $\omega_0\in H^1(\R^3)$ with compact support) gives rise to a strong solution on $[0,T^*]$ and $u_0^\alpha\in H^1({Q_\alpha})$ converges to $u^0$ in $H^1(\R^3)$, then for large enough $\alpha$ the equations on ${Q_\alpha}$ with initial data $u_\alpha^0$ give rise to a unique strong solution on the same interval, and $\tilde u_\alpha\to u$ as $\alpha\to\infty$ in $L^r(0,T;H^1(\R^3))$, $r\in[1,4)$. This shows that the existence of a regular solution on the whole space implies the existence of a regular solution on a large enough periodic domain.

The relationship between the existence of smooth solutions for the equations in various settings (peiodic boundary conditions, Schwartz solutions on $\R^3$, homogeneous and inhomogeneous problems) has also been considered, from a different point of view, by Tao (2013).

There are other `transfer of regularity' results for the Navier--Stokes equations in different contexts. Constantin (1988) showed that if $u_0\in H^{s+2}$, $s\ge 3$, gives rise to a solution in $L^\infty(0,T^*;H^{s+2})$ of the Euler equations, then for the Navier--Stokes equations with dissipative term $-\nu\Delta u$, one can take $\nu$ sufficiently small to ensure that the same initial condition produces an $H^s$-bounded solution of the Navier--Stokes equations on $[0,T^*]$. 
A variant of this approach in Chernyshenko et al.\ (2007) shows that if $u_0$ gives rise to a regular solution of the Navier--Stokes equations on $[0,T^*]$ then a sufficiently `good' numerical scheme will have a similarly smooth solution that will also exist on $[0,T^*]$. Other results that `transfer regularity' start with two-dimensional flows: Raugel \& Sell (1993) considered the problem posed on thin three-dimensional domains, and Gallagher (1997) considered flows with initial data that are `close to two dimensional'.

There is, of course, another way to view solving the equations on ${Q_\alpha}$, $\alpha\ge\alpha_0$, with fixed initial data $u_0$ of compact support. Here, rather than keeping $u_0$ fixed and increasing $\alpha$, one could keep the domain fixed and rescale $u_0$: taking $\alpha_0=1$ for simplicity, the problem on ${Q_\alpha}$ becomes a problem posed on $\Omega_1$ by setting
$$
u_0^\alpha(x)=\alpha u_0(\alpha x).
$$
A solution $(u(x,t),p(x,t))$ on ${Q_\alpha}$ becomes the rescaled solution $$(\alpha u(\alpha x,\alpha^2t),\alpha^2 p(\alpha x,\alpha^2t))$$ on $Q_1$. However, if the solution on ${Q_\alpha}$ exists for $t\in[0,T]$, then the rescaled solution on $Q_1$ exists only for $t\in[0,T/\alpha^2]$. It follows that such a rescaling is not a useful tool for considering the behaviour of solutions as $\alpha\to\infty$ in the sense proposed here. Nevertheless, related scaling ideas are used here to check that various inequalities hold with constants independent of the domain parameter $\alpha$.

\section{Preliminaries}

The expression $L^2(Q_\alpha)$ denotes the space of functions that are $2\alpha$-periodic in every direction, with
$$
\int_{Q_\alpha}|u|^2<\infty,
$$
where $Q_\alpha=(-\alpha,\alpha)^3$. Throughout the paper, a dot over a space denotes that the functions have zero average: so, for example, $\dot L^2(Q_\alpha)$ denotes that subset of $L^2(Q_\alpha)$ consisting of those functions that also satisfy the condition
\begin{equation}\label{zero-average}
\int_{{Q_\alpha}}u=0.
\end{equation}
The notation $\<f,g\>_{L^2(Q_\alpha)}=\int_{Q_\alpha} f(x)g(x)\,\d x$ is used for the inner product in $L^2(Q_\alpha)$.

The space of $2\alpha$-periodic functions with weak derivatives up to order $s$ in $L^2(Q_\alpha)$, again satisfying \eqref{zero-average}, is denoted by $\dot H^s({Q_\alpha})$. Due to the zero-average condition, the $\dot H^s(Q_\alpha)$ norm defined by setting
$$
\|u\|_{\dot H^s(Q_\alpha)}:=\left(\sum_{|\gamma|=s}\|\partial^\gamma u\|_{L^2(Q_\alpha)}^2\right)^{1/2}
$$
is equivalent to the full $H^s(Q_\alpha)$ norm. Indeed, for all $r\ge s\ge 0$ the generalised Poincar\'e inequality
$$
\|u\|_{\dot H^s(Q_\alpha)}\le C_{r,s}\alpha^{r-s}\|u\|_{\dot H^r(Q_\alpha)},\qquad u\in\dot H^r(Q_\alpha),
$$
holds, from which the equivalence follows.

Note also for later use that if $\Delta u\in L^2(Q_\alpha)$ then $u\in H^2(Q_\alpha)$ with
$$
\sum_{i,j=1}^3\|\partial_i\partial_ju\|_{H^2(Q_\alpha)}^2\le 9\|\Delta u\|_{L^2(Q_\alpha)}^2,
$$
since for any $f\in C^\infty(Q_1)$ with $f=\sum_{k\in\Z^3}\hat f_k\e^{\ri k\cdot x}$
  \begin{equation}\label{D2dd}
  \|\partial_i\partial_j f\|_{L^2(Q_1)}^2=\sum_{k\in\Z^3} |k_ik_j|^2|\hat f_k|^2\le \sum_{k\in\Z^3}|k|^4|\hat f_k|^2= \|\Delta f\|_{L^2(Q_1)}^2.
  \end{equation}

The notation $\dot C^\infty({Q_\alpha})$ denotes the space of all $C^\infty$ $2\alpha$-periodic functions satisfying the same zero average condition, and $\dot C_\sigma^\infty(Q_\alpha)$ the space of all smooth divergence-free functions in $\dot C^\infty(Q_\alpha)$. The space $\dot C_{c,\sigma}^\infty(\R^3)$ is the space of all smooth, compactly-supported, divergence-free functions defined on $\R^3$, with zero integral over $\R^3$. The space $\dot L^p_\sigma(Q_\alpha)$ is the completion of $\dot C_\sigma^\infty(Q_\alpha)$ in $L^p(Q_\alpha)$; similarly $\dot L^p_\sigma(\R^3)$ is the completion of $\dot C_{c,\sigma}^\infty(\R^3)$ in $L^p(\R^3)$. Throughout, the $\sigma$ subscript indicates that the functions are divergence free.

Note that $\dot C_\sigma^\infty(Q_\alpha)$ is dense in $\dot H^1_\sigma(Q_\alpha)$ and $\dot C_{c,\sigma}^\infty(\R^3)$ is dense in $H^1_\sigma(\R^3)$. The second of these two is less obvious so the proof is given here.

\begin{lemma}\label{surprise!}
$\dot C_{c,\sigma}^\infty(\R^3)$ is dense in $H^1_\sigma(\R^3)$.
\end{lemma}

\begin{proof}
 The density of $C_{c,\sigma}^\infty(\R^3)$ in $H^1_\sigma(\R^3)$ is due to Heywood (1976):  so given any $u\in H^1_\sigma(\R^3)$ and $\eps>0$, there exists $\phi\in C_{c,\sigma}^\infty(\R^3)$ such that $\|u-\phi\|_{H^1(\R^3)}<\eps/2$.

  Set $M=\int_{\R^3}\phi(x)\,\d x$ and choose any $\psi\in C_{c,\sigma}^\infty(\R^3)$ with $\int_{\R^3}\psi(x)\,\d x=1$. Setting $\psi_M(x):=MR^{-3}\psi(x/R)$ yields a $\psi_M\in C_{c,\sigma}^\infty(\R^3)$ with
  $$
  \int_{\R^3}\psi_M=M,\qquad\int_{\R^3}|\psi_M|^2=\frac{M^2}{R^3},\qquad\mbox{and}\qquad\int_{\R^3}|\nabla\psi_M|^2=\frac{M^2}{R^5}.
  $$
    Now choose $R$ sufficiently large that $\|\psi_M\|_{H^1}^2=M^2R^{-3}+M^2R^{-5}<\eps^2/4$; setting $\tilde u=\phi-\psi_M$ gives $\tilde u\in\dot C_{c,\sigma}^\infty(\R^3)$ with $\|u-\tilde u\|_{H^1(\R^3)}<\eps$.
\end{proof}

At various points it is important that the constants in inequalities valid on $Q_\alpha$ do not depend on $\alpha$, i.e.\ on the size of the domain. To ensure this, inequalities are shown on $Q_1$ and then rescaled: given a function $f_\alpha$ defined on $Q_\alpha$, the rescaled function $f(x)=f_\alpha(\alpha x)$ is defined on $Q_1$. The $L^p$ norms of derivatives of order $k$ then scale according to
\begin{equation}\label{rescaling}
\|\partial^\gamma f_\alpha\|_{L^p(Q_\alpha)}=\alpha^{(3/p)-k}\|\partial^\gamma f\|_{L^p(Q_1)},\qquad\mbox{where}\qquad |\gamma|=k.
\end{equation}

\section{Convergence of velocities corresponding to compactly-supported vorticity}\label{sec:vorticity}

\subsection{Reconstruction of $u$ from $\omega$}

One of the issues for the convergence results considered
here is to identify a class of initial data that is `localised' in a reasonable
way. One possible choice (although Theorem \ref{strong} is more general)
is to take a compactly supported vorticity $\omega$ and to consider the
corresponding velocity fields obtained by `inverting' the curl operator on the corresponding domain. This amounts to solving the equations
\begin{equation}\label{div-curl}
\curl\,u=\omega,\qquad\nabla\cdot u=0;
\end{equation}
by taking the curl of both equations and using the vector identity
$$
\curl\,\curl\,u=\nabla(\nabla\cdot u)-\Delta u=-\Delta u
$$
it follows that
$$
-\Delta u=\curl\,\omega\qquad\Rightarrow\qquad u=(-\Delta)^{-1}\curl\,\omega;
$$
the weak form of this system is: given $\omega\in\dot L^2_\sigma(\Omega)$,
\begin{equation}\label{div-curl-weak}
\mbox{find }u\in \dot H^1_\sigma(\Omega)\quad \mbox{s.t.}\quad \<\nabla u,\nabla\phi\>_{L^2(\Omega)}=\<\omega,\curl\,\phi\>_{L^2(\Omega)}\qquad\forall\ \phi\in \dot H^1_\sigma(\Omega),
\end{equation}
for $\Omega=Q_\alpha$, and replacing $\dot H^1_\sigma$ with in $H^1_\sigma$ (i.e.\ relaxing the zero average condition) on $\R^3$. Note the integration by parts in the right-hand side from $\<\curl\,\omega,\phi\>$, which allows for $\omega\in L^2$ and not only $\omega\in H^1$.


On the whole space, an expression for $u$ can be obtained using the fundamental solution of the Laplacian and an integration by parts, namely the Biot--Savart Law
\begin{equation}\label{BSL}
u=\curl^{-1}\omega:=-\frac{1}{4\pi}\int_{\R^3}\frac{x-y}{|x-y|^3}\times\omega(y)\,\d y;
\end{equation}
for $\omega\in L^{6/5}(\R^3)\cap L^2_\sigma(\R^3)$ this is the unique solution in $H^1_\sigma(\R^3)$ of \eqref{div-curl-weak}.

[In the case of $\R^2$ modified versions of the equivalent to the Biot--Savart Law are available which do not require decay of $\omega$ and $u$ at infinity, see Serfati (1995) and Ambrose et al.\ (2015), for example. For bounded domains see Enciso, Garc\'\i a-Ferrero, \& Peralta-Salas (2018), for example.]

On periodic domains, while $u_\alpha=\curl_\alpha^{-1}\omega$ can be written explicitly in terms of the Fourier expansion
 it will be more useful here to observe that $u_\alpha$ is still the solution of the equation $-\Delta u_\alpha=\curl\,\omega$.

On the periodic domain $Q_1$, if $\int_{Q_1}g=0$, then the equation $-\Delta u=g$, $\int_{Q_1}u=0$, has a solution given in the form
$$
u(x)=\int_{Q_1} K_Q(x,y)g(y)\,\d y,\quad\mbox{with}\quad K_Q(x,y)=\frac{1}{|x-y|}\phi(x-y)+S(x,y),
$$
where $\phi$ and $S$ are smooth and $\phi(z)=1$ for $|z|<1/10$ and $\phi(z)=0$ for $|z|>1/4$, see Theorem C.5 in Robinson et al.\ (2016), for example. Then, when $\omega$ has compact support in $Q_1$, \begin{align}
u(x)&=\int_{Q_1}\left[\frac{1}{|x-y|}\phi(x-y)+S(x,y)\right][\curl\,\omega](y)\,\d y\notag\\
&=\int_{Q_1}\frac{1}{|x-y|}\phi(x-y)[\curl\,\omega](y)\,\d y+\int_{Q_1}S(x,y)[\curl\,\omega](y)\,\d y\notag\\
&=\int_{Q_1}\curl_y\left(\frac{1}{|x-y|}\phi(x-y)\right)\omega(y)\,\d y+\int_{Q_1}[\curl_yS](x,y)\omega(y)\,\d y\notag\\
&=-\int_{Q_1}\phi(x-y)\frac{x-y}{|x-y|^3}\times\omega(y)\,\d y+\int_{Q_1}\frac{1}{|x-y|}\nabla\phi(x-y)\times\omega(y)\,\d y\notag\\
&\qquad\qquad+\int_{Q_1}[\curl_yS](x,y)\omega(y)\,\d y.\label{BS-Qa}
\end{align}

\subsection{Bounds on $u$ from bounds on $\omega$}

The following result is extremely useful; it is valid on $Q_\alpha$ for every $\alpha$ and on $\R^3$. While a similar inequality could be obtained using the Calder\'on--Zygmund Theorem and \eqref{BSL}, equality follows here from a much simpler argument (see equation (1.4.20) in Doering \& Gibbon, 1995).

\begin{lemma}\label{lem:w=Du} If $u\in H^1_\sigma$ and $\omega={\rm curl}\,u\in L^2$ then $\|\nabla u\|_{L^2}=\|\omega\|_{L^2}$.
\end{lemma}

\begin{proof}
Assume first that $u$ is smooth and $\omega\in L^2$. Then, since $\omega_i=\epsilon_{ijk}\partial_ju_k$ and $\epsilon_{ijk}\epsilon_{ilm}=\delta_{jl}\delta_{km}-\delta_{jm}\delta_{kl}$,
\begin{align*}
  \int|\omega|^2&=\int \epsilon_{ijk}(\partial_ju_k)\epsilon_{ilm}(\partial_lu_m)\\
  &=\int[\delta_{jl}\delta_{km}-\delta_{jm}\delta_{kl}](\partial_ju_k)(\partial_lu_m)\\
  &=\int(\partial_ju_k)(\partial_ju_k)-(\partial_ju_k)(\partial_ku_j)=\int\sum_{j,k}|\partial_ju_k|^2,
\end{align*}
integrating by parts twice in the final term and using the fact that $u$ is divergence free. Now if $u\in H^1$, $\omega\in L^2$ and mollifying $u$ produces a smooth $u_\eps$ with $\nabla\times u_\eps\in L^2$; the same argument shows that since $\omega_\eps\to\omega$, $\partial_i(u_\eps)_j\to\partial_iu_j$ for every $i,j$, yielding the same equality for these more general $u$.\end{proof}

The Biot--Savart Law and Young's inequality provide $L^q$ estimates on $u$ given $L^p$ bounds on $\omega$.

\begin{lemma}\label{lem:w-unif}
  Suppose that $\omega\in L_\sigma^{p}({\R^3})$ for some $p\in(1,3)$. Then, for
  $$
  \frac{1}{q}=\frac{1}{p}-\frac{1}{3},
  $$
 $u=\curl^{-1}\omega\in L_\sigma^q({\R^3})$ with
   \begin{equation}\label{w2u-R3}
  \|u\|_{L^q({\R^3})}\le C_p\|\omega\|_{L^p({\R^3})}.
  \end{equation}
  The same estimate also holds when $\omega\in \dot L_\sigma^p(Q_\alpha)$: $u_\alpha=\curl_\alpha^{-1}\omega\in \dot L^q_\sigma(Q_\alpha)$ with
    \begin{equation}\label{w2u-Qa}
  \|u_\alpha\|_{L^q({Q_\alpha})}\le C_p\|\omega\|_{L^p({Q_\alpha})},
  \end{equation}
  where $C_p$ is independent of $\alpha$.
  \end{lemma}

\begin{proof}
On the whole space $u$ is given by \eqref{BSL}. So $u$ is given by the convolution of $\omega$ with a kernel of order $|x|^{-2}$; in three dimensions this belongs to the weak Lebesgue space $L^{3/2,\infty}$, and \eqref{w2u-R3} follows using the weak-Lebesgue space version of Young's inequality,
$$
\|f\star g\|_{L^q}\le C_{p,q,r}\|f\|_{L^{r,\infty}}\|g\|_{L^p},\qquad 1+\frac{1}{q}=\frac{1}{r}+\frac{1}{p},\qquad 1<p,q,r<\infty.
$$

For the same bound on $Q_1$, consider the expression in \eqref{BS-Qa},
\begin{align*}
u(x)&=-\int_{Q_1}\phi(x-y)\frac{x-y}{|x-y|^3}\times\omega(y)\,\d y+\int_{Q_1}\frac{1}{|x-y|}\nabla\phi(x-y)\times\omega(y)\,\d y\\
&\qquad\qquad+\int_{Q_1}[\curl_yS](x,y)\omega(y)\,\d y.
\end{align*}
The kernel in the first term is once again in $L^{3/2,\infty}(Q_1)$ and the kernel in the second term is in $L^{3/2}(Q_1)$; these two terms are thus bounded in $L^q(Q_1)$ using Young's inequality. For the final term $u_3(x)$, Minkowski's inequality yields
$$
\|u_3\|_{L^q(Q_1)}\le\int_{Q_1}\|\curl_yS(\cdot,y)\|_{L^q(Q_1)}|\omega(y)|\,\d y.
$$
Noting that $S$ is smooth and that only $x,y\in Q_1$ are relevant, the bound $\|\curl_yS_\alpha(\cdot,y)\|_{L^q(Q_1)}\le M$ holds, and hence
$$
\|u_3\|_{L^q(Q_1)}\le M\int_{Q_1}|\omega(y)|\,\d y\le M\|\omega\|_{L^1(Q_1)}\le M_p\|\omega\|_{L^p(Q_1)},
$$
using H\"older's inequality and the fact that $Q_1$ is bounded.

These three upper bounds combine to yield \eqref{w2u-Qa} on $Q_1$. The fact that the same inequality holds with a constant independent of $\alpha$ follows since both norms in \eqref{w2u-Qa} behave the same way under the rescaling $x\mapsto\alpha x$, see \eqref{rescaling}.\end{proof}
%

\subsection{Extension of functions from ${Q_\alpha}$ to $\R^3$}

Given $\omega\in\dot L^2_\sigma(\R^3)$ with support contained in $Q_{\alpha_0}$, Lemma \ref{lem:w-unif} gives a family $\{u_\alpha\}_{\alpha\ge\alpha_0}$ of velocity fields defined on $Q_\alpha$ ($\alpha\ge\alpha_0$). In order to be able to take a meaningful limit on the whole of $\R^3$, each $u_\alpha$ will be extended to the whole of $\R^3$ in such a way that the support of $\tilde u_\alpha$ is contained in a domain only slightly larger than ${Q_\alpha}$.

Given $u_\alpha\in L^2({Q_\alpha})$, denote by $\tilde u_\alpha$ the extension of $u_\alpha$ to all of $\R^3$ defined by setting
  $$
  \tilde u_\alpha(x)=\psi_\alpha(x)u_\alpha^{\rm p}(x),
  $$
  where $u_\alpha^{\rm p}(x)$ is the periodic extension of $u_\alpha$ to $\R^3$ and $\psi_\alpha\in C_c^\infty(\R^3)$ with $0\le\psi_\alpha\le 1$,
  $$
  \psi_\alpha(x)=\begin{cases}1&x\in(-\alpha,\alpha)^3\\
  0&x\notin(-(\alpha+1),\alpha+1)^3,\end{cases}
  $$
  $|\nabla\psi_\alpha|\le M_1$, and $|\nabla^2\psi_\alpha|\le M_2$, uniformly in $\alpha$.

Bounds on $u_\alpha$ immediately translate to bounds on $\tilde u_\alpha$: in particular, for $\alpha\ge 1$,
  $$
  \|\tilde u_\alpha\|_{L^2(\R^3)}\le e_1\|u\|_{L^2({Q_\alpha})},\qquad\|\nabla\tilde u_\alpha\|_{L^2(\R^3)}\le e_2\|u\|_{H^1({Q_\alpha})},\qquad
  $$
  and
  $$
  \|\tilde u_\alpha\|_{H^2(\R^3)}\le e_3\|u\|_{H^2({Q_\alpha})}
  $$
  [for explicit values of these constants, one can take $e_1=27$, $e_2=\max(26M_1,27)$, and $e_3=\max(27M_2,52M_1,27)$].

  Later a similar extension will be used for time-dependent functions $u_\alpha(x,t)$; in this case
  $$
  \tilde u_\alpha(x,t):=\psi_\alpha(x)u_\alpha^{\rm p}(x,t),
  $$
  with the cut-off function $\psi_\alpha$ being independent of $t$. This means, in particular, that
  $$
  \partial_t\tilde u_\alpha(x,t)=\psi_\alpha(x)[\partial_tu_\alpha]^{\rm p}(x,t),
  $$
  so that bounds on $\partial_t\tilde u_\alpha$ can be deduced from bounds on $\partial_tu_\alpha$ as done for $\tilde u_\alpha$ above.

  \subsection{Convergence of $\curl_\alpha^{-1}\omega$ to $\curl^{-1}\omega$ as $\alpha\to\infty$}

Theorem \ref{wa2w} will show that the fields $\tilde u_\alpha$ from Lemma \ref{lem:w-unif} converge to $u$ strongly in $H^1(\R^3)$ whenever $\omega\in H^1(\R^3)$.   The following lemma (see Leray, 1934, or Lemma 6.34 in O\.za\'nski \& Pooley, 2018) can be used to improve the $L^2$-convergence of $\tilde u_\alpha$ to $u$ on compact subsets of $\R^3$ to convergence on the whole of $\R^3$ by bounding the `tails' of $u_\alpha$ uniformly.

  \begin{lemma}\label{Leray-idea}
  If $\{f_\alpha\}_{\alpha\ge\alpha_0},f\in L^2(\R^3)$; $f_\alpha\to f$ strongly in $L^2(K)$ for every compact subset $K$ of $\R^3$; and for every $\eta>0$ there exist $R(\eta)$ and $\beta(\eta)$ such that
  \begin{equation}\label{Leray}
  \int_{|x|\ge R}|f_\alpha|^2<\eta\qquad\mbox{for all }\alpha\ge\beta,
  \end{equation}
  then $f_\alpha\to f$ in $L^2(\R^3)$.
  \end{lemma}
%

The argument that follows obtains bounds on the `tail' of a sequence $u_\alpha\in L^2({Q_\alpha})$; in order to apply Lemma \ref{Leray-idea} the corresponding bounds on $\tilde u_\alpha$ will be needed. Therefore note here that if $u_\alpha\in L^2({Q_\alpha})$ and $R<\alpha-1$ then
  \begin{equation}\label{mest}
  \int_{|x|\ge R}|\tilde u_\alpha|^2\,\d x\le 27\int_{x\in {Q_\alpha}:\ |x|\ge R}|u_\alpha|^2\,\d x,
  \end{equation}
since
$$
\bigcup_{\underline{k}\in\Z^3}B(2\alpha\underline{k},R)\cap{\rm supp}(\tilde u_\alpha)=B(0,R),
$$
i.e.\ the integral on the left-hand side of \eqref{mest} can at most include the `tails' from the periodic cells immediately adjacent to $Q_\alpha$, see Figure \ref{fig:2D} for an illustration of this in the two-dimensional case, where the corresponding constant is 9. [In 2D this can be improved to 4; following a similar idea the constant in the 3D case can be improved to 10.]

\begin{figure}
\begin{centering}
  \includegraphics[scale=0.5]{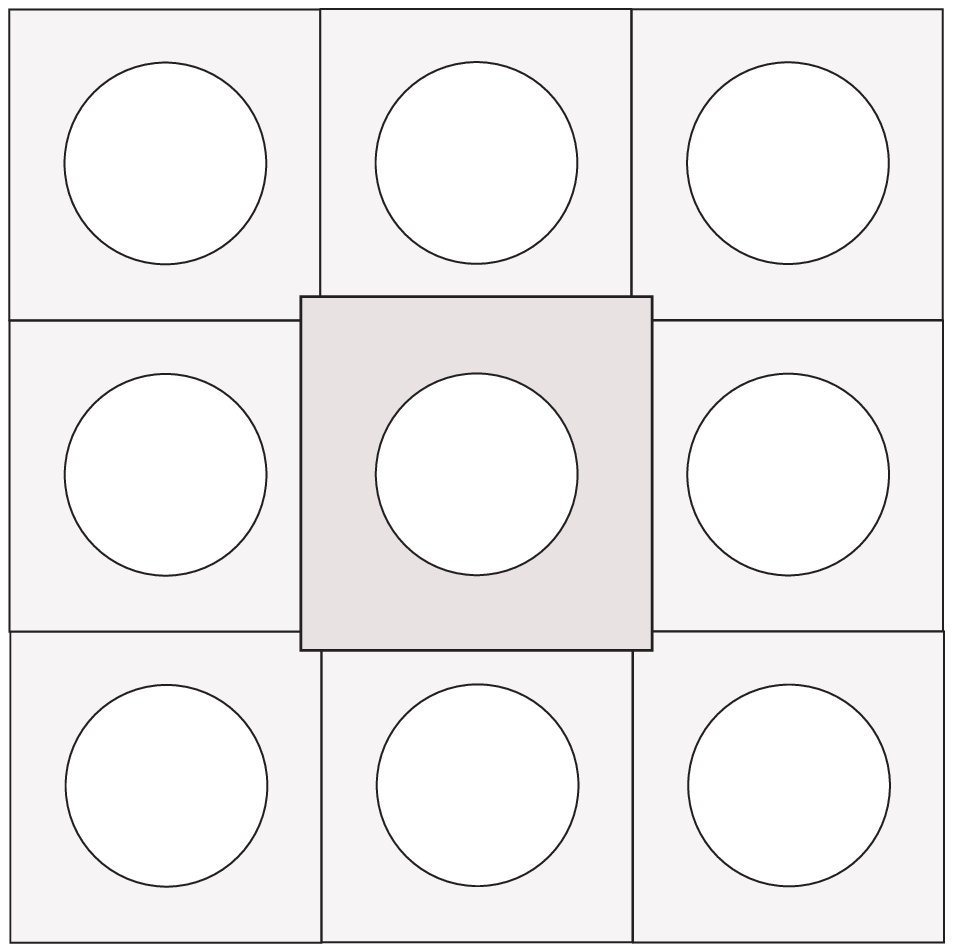}\qquad\includegraphics[scale=0.5]{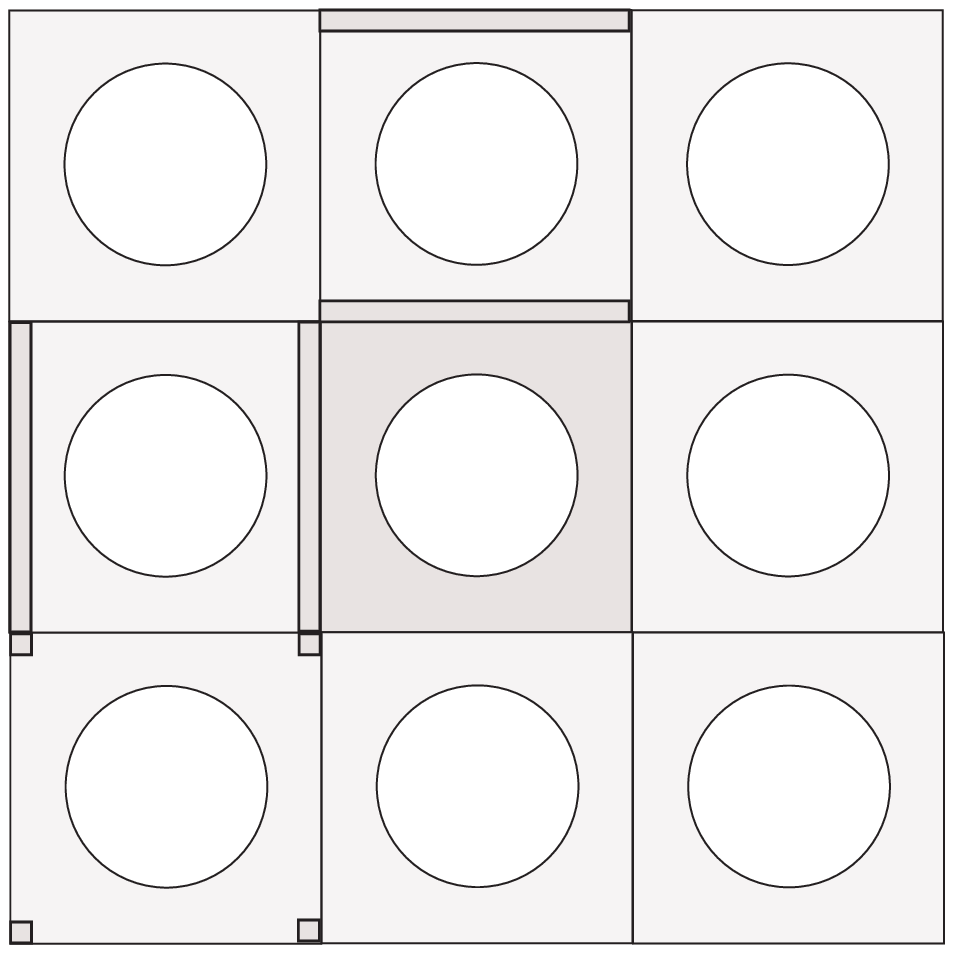}
  \begin{caption}{The support of $\tilde u_\alpha$ is contained in the large central square in the left-hand figure, and $|\tilde u_\alpha|\le|u_\alpha|$ everywhere. Periodised circles of radius $R$ are shown in white. Clearly $\int_{|x|\ge R}|\tilde u_\alpha|^2\le 9\int_{x\in Q_\alpha:\ |x|\ge R}|u_\alpha|^2$. However, with portions of this darker square moved using periodicity (on the right) this can be improved to $\int_{|x|\ge R}|\tilde u_\alpha|^2\le 4\int_{x\in Q_\alpha:\ |x|\ge R}|u_\alpha|^2$. }\end{caption}\label{fig:2D}
  \end{centering}
\end{figure}

\begin{theorem}\label{wa2w}
Suppose that $\omega\in \dot L^2_\sigma(\R^3)$ has compact support. For every $\alpha$ sufficiently large that ${\rm supp}(\omega)\subset Q_\alpha$ define $u_\alpha=\curl_\alpha^{-1}\omega$. Then
\begin{equation}\label{w2u-unif}
\|u_\alpha\|_{L^2}\le C\|\omega\|_{L^{6/5}},\qquad\|\nabla u_\alpha\|_{L^2}=\|\omega\|_{L^2},
\end{equation}
$\tilde u_\alpha\cgw\curl^{-1}\omega$ weakly in $H^1(\R^3)$ and $\tilde u_\alpha\to\curl^{-1}\omega$ strongly in $L^2(K)$ for every compact subset $K$ of $\R^3$.

If in addition $\omega\in H^1(\R^3)$ then $u_\alpha\in H^2(\R^3)$, $\tilde u_\alpha\cgw\curl^{-1}\omega$ weakly in $H^2(\R^3)$, and $\tilde u_\alpha\to\curl^{-1}\omega$ strongly in $H^1(\R^3)$.
\end{theorem}

\begin{proof}
If $\omega\in L^2$ then the uniform estimates for $u_\alpha$ in \eqref{w2u-unif} follow from Lemmas \ref{lem:w=Du} and \ref{lem:w-unif}. Now extend each $u_\alpha$ to a function $\tilde u_\alpha$ defined on all of $\R^3$ as outlined above, and in this way obtain a set of functions with $\tilde u_\alpha$ uniformly bounded (with respect to $\alpha$) in $H^1(\R^3)$. Since $H^1(\R^3)$ is reflexive, it follows from reflexive weak sequential compactness that there exists an element $v\in H^1(\R^3)$ such that $\tilde u_{\alpha_j}\cgw v$ weakly in $H^1(\R^3)$, which in turn implies the strong convergence in $L^2(K)$ for every compact subset $K$ of $\R^3$.

It remains to show that $v=u:=\curl^{-1}\omega$ and that the convergence takes place as $\alpha\to\infty$ and not just for a subsequence.

   To this end, take $\varphi\in\dot C_{c,\sigma}^\infty(\R^3)$. Then, since $\tilde u_{\alpha_j}=u_{\alpha_j}$ on $Q_{\alpha_j}$, once ${\rm supp}(\varphi)\subset Q_{\alpha_j}$ we have
  $$
  \<\nabla\tilde u_{\alpha_j},\nabla\varphi\>_{L^2(\R^3)}=\<\nabla u_{\alpha_j},\nabla\varphi\>_{L^2(Q_{\alpha_j})}=\<\omega,\curl\,\varphi\>_{L^2(\Omega_{\alpha_j})}=\<\omega,\curl\,\varphi\>_{L^2(\R^3)}.
    $$

  Since $\nabla\tilde u_{\alpha_j}\cgw \nabla u$ weakly in $L^2(\R^3)$, for each fixed $\varphi$ it follows that
  $$
  \<\nabla u,\nabla\varphi\>_{L^2(\R^3)}=\<\omega,\curl\,\varphi\>_{L^2(\R^3)}
  $$
  for every $\varphi\in\dot C_{c,\sigma}^\infty(\R^3)$; the equality then holds for every $\varphi\in H^1_\sigma(\R^3)$ by density (see Lemma \ref{surprise!}). Since $u\in H^1_\sigma(\R^3)$ it follows that $u$ is the unique $H^1$ solution of $-\Delta u=\curl\,\omega$, which is precisely $\curl^{-1}\omega$. This also shows that the limit of any convergent subsequence must be the same, and it follows that $u_\alpha\to u$ as claimed in the statement of the theorem.

  If in addition $\omega\in H^1(\R^3)$ then standard elliptic regularity results (see Evans, 2010, for example) gives uniform estimates on $\tilde u_\alpha$ in $H^2(\R^3)$, since then
  $$
  \|\Delta u_\alpha\|_{L^2(Q_\alpha)}=\|\curl\,\omega\|_{L^2(Q_\alpha)}
  $$
  and this yields a bound on the other second derivatives, see \eqref{D2dd}. The weak convergence in $H^2(\R^3)$ now follows since $H^2$ is reflexive, which implies the strong convergence in $H^1(K)$ for every compact subset $K$ of $\R^3$.

To improve this to strong convergence in $H^1(\R^3)$, take $\phi=u_\alpha\varrho_\alpha$ as the test function in
$$
\<\nabla u_\alpha,\nabla\phi\>=\<\curl\,\omega,\phi\>
$$
(cf.\ \eqref{div-curl-weak}),
where $\varrho_\alpha$ is the restriction of
  \begin{equation}\label{eq:varrho}
  \varrho=\begin{cases}0&|x|<r\\
  \frac{|x|-r}{R-r}&r\le |x|\le R\\
  1&|x|>R,\end{cases}
  \end{equation}
  to ${Q_\alpha}$, where we take $0<r<R<\alpha$; 
  note that
  $$
  |\nabla\varrho_\alpha|=\begin{cases}0&|x|<r\\\frac{1}{R-r}&r<|x|<R\\0&|x|>R.\end{cases}
  $$
Therefeore
  $$
   \int_{{Q_\alpha}}|\nabla u_\alpha|^2\varrho_\alpha=-\int_{{Q_\alpha}}(\nabla u_\alpha)\cdot(\nabla\varrho_\alpha)u_\alpha+\int_{{Q_\alpha}} (\curl\,\omega)u_\alpha\varrho_\alpha,
   $$
and taking $r$ sufficiently large that ${\rm supp}(\omega)\subset B(0,r)$ yields
   \begin{align*}
   \int_{x\in {Q_\alpha}:\ |x|\ge R}|\nabla u_\alpha|^2&\le\frac{1}{R-r}\|\nabla u_\alpha\|_{L^2({Q_\alpha})}\|u_\alpha\|_{L^2({Q_\alpha})}\\
   &\le\frac{K}{R-r}\|\omega\|_{L^{6/5}}\|\omega\|_{L^2}.
   \end{align*}
  Lemma \ref{Leray-idea} now guarantees that $\nabla\tilde u_\alpha\to\nabla u$ in $L^2(\R^3)$.

   It remains to show that $\tilde u_\alpha\to u$ in $L^2(\R^3)$. First, since in 3D the Sobolev embedding  $\|f\|_{L^6(\R^3)}\le C\|\nabla f\|_{L^2(\R^3)}$ holds for $f\in H^1(\R^3)$, it follows that
   $$
   \|\tilde u_\alpha-u\|_{L^6(\R^3)}\le C\|\nabla\tilde u_\alpha-\nabla u\|_{L^2(\R^3)},
   $$
   and so $\tilde u_\alpha\to u$ in $L^6(\R^3)$. Now, since $\omega\in L^{24/23}(Q_\alpha)$, Lemma \ref{lem:w-unif} implies that
   $$
   \|u_\alpha\|_{L^{8/5}(Q_\alpha)}\le K\|\omega\|_{L^{24/23}(Q_\alpha)},
   $$
   a bound that holds uniformly in $\alpha$ and yields a similar uniform bound on $\tilde u_\alpha$ in $L^{8/5}(\R^3)$. Finally, the Lebesgue interpolation
   $$
   \|\tilde u_\alpha-u\|_{L^2(\R^3)}\le\|\tilde u_\alpha-u\|_{L^{8/5}(\R^3)}^{8/11}\|\tilde u_\alpha-u\|_{L^6(\R^3)}^{3/11}
   $$
   guarantees that $\tilde u_\alpha\to u$ in $L^2(\R^3)$.

   Combining the convergence of $\tilde u_\alpha\to u$ and $\nabla \tilde u_\alpha\to\nabla u$ in $L^2(\R^3)$ shows that $\tilde u_\alpha\to u$ in $H^1(\R^3)$ as claimed.
  \end{proof}


\section{Weak and strong solutions of the Navier--Stokes equations}\label{sec:NSE}

For $\Omega=Q_\alpha$ or $\R^3$, denote by $\D_\sigma(\Omega)$ the space of all test functions on $\Omega\times[0,\infty)$ given by
$$
\D_\sigma(\Omega)=\{\phi\in C_c^\infty(\Omega\times[0,\infty)):\ \nabla\cdot\phi(t)=0\mbox{ for all }t\in[0,\infty)\}.
$$

\begin{definition}\label{def-weak}
A function $u$ is a weak solution of the Navier--Stokes equations corresponding to the initial condition $u_0\in\dot L^2_\sigma(\Omega)$ if
$$
u\in L^\infty(0,T;\dot L^2_\sigma(\Omega))\cap L^2(0,T;H^1(\Omega))\qquad\mbox{for every}\qquad T>0
$$
and
$$
\int_0^\infty-\<u,\partial_t\phi\>+\int_0^\infty\<\nabla u,\nabla\phi\>+\int_0^\infty\<(u\cdot\nabla)u,\phi\>=\<u_0,\phi(0)\>
$$
for all test functions $\phi\in{\mathcal D}_\sigma(\Omega)$.
\end{definition}

The following theorem combines the basic existence result for weak solutions (Leray, 1934; Hopf, 1951) with the property that at least one solution exists that satisfies the strong energy inequality (Leray, 1934; Ladyzhenskaya, 1969): see Theorems 4.4, 4.6, 4.10, and 14.4 in Robinson et al.\ (2016).

\begin{theorem}\label{weak-ex}
For every initial condition $u_0\in\dot L^2_\sigma(\Omega)$ there exists at least one global-in-time weak solution $u$ of the Navier--Stokes equations on $\Omega$ that satisfies the strong energy inequality
\begin{equation}\label{SEI}
\frac{1}{2}\|u(t)\|_{L^2(\Omega)}^2+\int_s^t\|\nabla u\|_{L^2(\Omega)}^2\le\frac{1}{2}\|u(s)\|_{L^2(\Omega)}^2\qquad\mbox{for all }t>s
\end{equation}
for almost all times $s\in[0,\infty)$, including $s=0$. [These are known as Leray--Hopf weak solutions.]
\end{theorem}

Note that it follows from this definition that any weak solution $u$ has a weak time derivative $\partial_tu$ with
$$
\partial_tu\in L^{4/3}(0,T;H^{-1}_\sigma(\Omega))\qquad\mbox{for every }T>0,
$$
where $H^{-1}_\sigma(\Omega)$ is the dual space of $\dot H^1_\sigma(\Omega)$, with
\begin{equation}\label{dtu43}
\|\partial_tu\|_{L^{4/3}(0,T;H^{-1}_\sigma(\Omega))}\le c\int_0^T\|\nabla u\|^{2}\|u\|^{2/3}+T^{1/3}\left(\int_0^T\|\nabla u\|^2\right)^{2/3},
\end{equation}
with $c$ independent of $\alpha$; see Lemma 3.7 in Robinson et al.\ (2016).

Key to later results in this paper is the notion of a strong solution.

\begin{definition}\label{def-strong}
A function $u$ is a strong solution on $[0,T]$ of the Navier--Stokes equations corresponding to the initial condition $u_0\in \dot H^1_\sigma(\Omega)$ if it is a weak solution and in adddition
$$
u\in L^\infty(0,T;H^1(\Omega))\cap L^2(0,T;H^2(\Omega)).
$$
\end{definition}

The following theorem on the existence of strong solutions is again valid on $Q_\alpha$ and $\R^3$; the constant $c$ is the same for all these domains. The result as stated combines Theorems 6.4, 6.8, 6.15, and 7.5 in Robinson et al.\ (2016).

\begin{theorem}\label{thm:strong}
Any initial condition $u_0\in \dot H^1_\sigma(\Omega)$ gives rise to a unique strong solution of the Navier--Stokes equations at least on the time interval $[0,T]$, where $T=c\|\nabla u_0\|_{L^2(\Omega)}^{-4}$. For such solutions the equation
$$
\partial_tu-\Delta u+(u\cdot\nabla)u+\nabla p=0
$$
is satisfied as an equality in $L^2(0,T;L^2(\Omega))$, and in fact $u$ is smooth in space-time on $\Omega\times(0,T]$.
\end{theorem}

\section{Convergence of weak solutions}

Convergence of weak solutions as $\alpha\to\infty$ is relatively straightforward; indeed, a similar method has been used by Heywood (1988; see also Theorem 4.10 in Robinson et al., 2016) to prove the existence of weak solutions on the whole space, although with that aim it is probably more natural to consider the equations with Dirichlet boundary conditions on the domains $B(0,\alpha)$, which can easily be extended by zero to all of $\R^3$.

\begin{proposition}\label{prop:conv1-weak}
  Suppose that $u_\alpha^0\in \dot L^2_\sigma(Q_\alpha)$ with $\tilde u_\alpha^0\cgw u^0$ in $L^2(\R^3)$. Let $u_\alpha$ be weak solutions of the equations on $Q_\alpha$ with initial conditions $u_\alpha^0$ that satisfy the energy inequality
  \begin{equation}\label{wk-EI}
  \frac{1}{2}\|u_\alpha(t)\|_{L^2(Q_\alpha)}^2+\int_0^t\|\nabla u_\alpha(s)\|_{L^2(Q_\alpha)}^2\,\d s\le\frac{1}{2}\|u_\alpha^0\|_{L^2(Q_\alpha)}^2
  \end{equation}
  for almost every $t>0$. Then there exists a  weak solution $u$ of the equations on $\R^3$, and a subsequence $u_{\alpha_j}$ such that, for every $T>0$, $\tilde u_{\alpha_j}$ converges to $u$ weakly in $L^2(0,T;H^1)$ and strongly in $L^2(0,T;L^2(K))$ for every compact subset $K$ of $\R^3$.
\end{proposition}

\begin{proof}
Since $\tilde u_\alpha^0$ is a weakly-convergent sequence it must be bounded in $L^2(\R^3)$; so $u_\alpha^0$ is uniformly bounded in $L^2(Q_\alpha)$, and it is immediate from the energy inequality \eqref{wk-EI} that $u_\alpha$ is uniformly bounded (with respect to $\alpha$) in $L^\infty(0,T;L^2(Q_\alpha))$ and $L^2(0,T;H^1(Q_\alpha))$. The inequality \eqref{dtu43} also provides uniform bounds on the time derivative $\partial_tu_\alpha$ in $L^{4/3}(0,T;H^{-1}_\sigma(Q_\alpha))$.

These uniform bounds on $u_\alpha$ become uniform bounds on the extended functions $\tilde u_\alpha$ in $L^\infty(0,T;L^2(\R^3))$ and $L^2(0,T;L^2(\R^3))$, so there exists an element  $u\in L^\infty(0,T;L^2_\sigma(\R^3))\cap L^2(0,T;H^1(\R^3))$ and a subsequence $\tilde u_{\alpha_j}$ that converges to $u$ weakly-$*$ in $L^\infty(0,T;L^2(\R^3))$ and for which
$$
\nabla \tilde u_{\alpha_j}\cgw \nabla u\qquad\mbox{in}\qquad L^2(0,T;L^2(\R^3)).
$$

However, it is not necessarily the case that $\partial_t\tilde u_\alpha$ is uniformly bounded in $L^{4/3}(0,T;H_\sigma^{-1}(\R^3))$, since there is no reason why the restriction of a `test function' $\phi\in\dot H^1_\sigma(\R^3)$ to $Q_\alpha$ should respect the periodic boundary conditions or integrate to zero, i.e.\ be an element of $\dot H^1_\sigma(Q_\alpha)$. To obtain strong convergence in $L^2(0,T;L^2(K))$ for compact subsets $K$ of $\R^3$, instead observe that for each $R>0$, once $\alpha>R$
$$
(\partial_t\tilde u_\alpha)|_{B(0,R)}=(\partial_tu_\alpha)|_{B(0,R)},
$$
and that if $\alpha>3R$ then any $\phi\in H^1_{0,\sigma}(B(0,R)):=H^1_0(B(0,R))\cap L^2_\sigma(B(0,R))$ can be extended to an element $\hat\phi\in \dot H^1_\sigma(Q_\alpha)$ with
$$
\|\hat\phi\|_{H^1(Q_\alpha)}=2\|\phi\|_{H^1(B(0,R))},
$$
by setting
$$
\hat\phi(x)=\begin{cases}
\phi(x)&x\in B(0,R)\\
-\phi(x)&x\in B((2R,0,0),R)\\
0&\mbox{otherwise};
\end{cases}
$$
the part of the extension where $\hat\phi(x)=-\phi(x)$ ensures that $\int_{Q_\alpha}\hat\phi=0$. It follows that once $\alpha>3R$,
$$
\|\partial_t\tilde u_\alpha\|_{H_{0,\sigma}^{-1}(B(0,R))}\le 2\|\partial_tu_\alpha\|_{H^{-1}_\sigma(Q_\alpha)}.
$$
It is also clear that
$$
\|\tilde u_\alpha\|_{H^1(B(0,R))}\le\|u_\alpha\|_{H^1(Q_\alpha)},
$$
so $\tilde u_\alpha$ is uniformly bounded in $L^2(0,T;H^1_\sigma(Q_\alpha))$. An application of the Aubin--Lions compactness theorem (see Simon, 1987) now yields a subsequence that converges strongly in $L^2(0,T;L^2(B(0,R)))$ for every $R>0$, and hence in $L^2(0,T;L^2(K))$ for every compact subset $K$ of $\R^3$.

%
%
%

It remains only to show that $u$ is a solution of the equations on the whole space.

To do this, take any test function $\phi\in\D_\sigma(\R^3)$ and let $M$ and $T$ be large enough that the support of $\phi$ is contained in $Q_M\times[0,T)$. Then for all $\alpha\ge M$ it follows from Definition \ref{def-weak}, since $\tilde u_\alpha=u_\alpha$ on $Q_\alpha$, that
  $$
  -\int_0^\infty\<\tilde u_{\alpha_j},\partial_t\phi\>+\int_0^\infty\<\nabla \tilde u_{\alpha_j},\nabla\phi\>+\int_0^\infty\<(\tilde u_{\alpha_j}\cdot\nabla)\tilde u_{\alpha_j},\phi\>=\<\tilde u_{\alpha_j}^0,\phi(0)\>.
  $$
  Passing to the limit as $j\to\infty$ -- using the weak convergence of gradients, the strong convergence in $L^2(0,T;L^2(\Omega_M))$, and the fact that $\tilde u^0_{\alpha_j}\cgw u^0$ -- shows that $u$ is a weak solution of the equations on $\R^3$ with initial condition $u^0$, as required.\end{proof}

  Note that the above proof does not show that the solution $u$ on $\R^3$ satisfies the energy inequality; this is why the limiting procedure here is not the ideal way to generate solutions of the equations on $\R^3$.

  \section{Convergence of strong solutions}

 The main result of this paper, Theorem \ref{strong}, will show that given a suitably convergent family of initial data $u_\alpha^0\in H^1(Q_\alpha)$, the `solutions' $\tilde u_\alpha$ converge strongly to $u$ in $L^2(0,T;H^1(\R^3))$.


\subsection{Uniform inequalities}

Key to obtaining uniform estimates for strong solutions on expanding domains are the following inequalities.

\begin{lemma}[Uniform inequalities]\label{goodC}
  There exist constants $C_A$ and $C_6$, which do not depend on $\alpha$, such that
  \begin{equation}\label{ineq-noa}
  \left\|u\right\|_{L^\infty({Q_\alpha})}\le C_A\|\nabla u\|_{L^2({Q_\alpha})}^{1/2}\|\Delta u\|_{L^2({Q_\alpha})}^{1/2}\qquad\mbox{for all }u\in\dot H^2({Q_\alpha}),
  \end{equation}
  and
  \begin{equation}\label{L6ineq}
  \|u\|_{L^6({Q_\alpha})}\le C_6\|\nabla u\|_{L^2({Q_\alpha})},\qquad\mbox{for all }u\in \dot H^1({Q_\alpha}).
  \end{equation}
  If $-\Delta p=\nabla\cdot[(u\cdot\nabla)u]$ with $\int_{{Q_\alpha}}p=\int_{{Q_\alpha}}u=0$ then
    \begin{equation}\label{CZ-p}
  \|p\|_{L^2({Q_\alpha})}\le C_Z\|u\|_{L^4({Q_\alpha})}^2.
  \end{equation}
  where $C_Z$ is independent of $\alpha$.
\end{lemma}

\begin{proof}
  The validity of the estimate \eqref{ineq-noa} for a fixed value of $\alpha$ is standard, and follows by splitting the Fourier series expansion of $u$ into `low modes' and `high modes' (see Exercise 1.10 in Robinson et al.\ (2016), for example): so, taking $\alpha=1$, for all $v\in H^2(\Omega_1)$
  $$
  \|v\|_{L^\infty(\Omega_1)}\le C_A\|\nabla v\|_{L^2(\Omega_1)}^{1/2}\|\Delta v\|_{L^2(\Omega_1)}^{1/2}.
  $$
  The rescalings in \eqref{rescaling} now show
  that this inequality is valid with the same constant on ${Q_\alpha}$.


   Inequality \eqref{L6ineq} in the case $\alpha=1$ is a consequence of the embedding $H^1(Q_1)\subset L^6(Q_1)$ valid for three-dimensional domains, and the Poincar\'e inequality $\|u\|_{L^2(Q_1)}\le C_P\|\nabla u\|_{L^2(Q_1)}$ which holds when $\int_{Q_1}u=0$. A similar rescaling argument shows that the same constant works for every $\alpha$.

   Finally, on $Q_1$, the estimate \eqref{CZ-p} follows using the Calder\'on--Zygmund Theorem,
   \begin{equation}\label{CZQ1}
   \|p\|_{L^2(Q_1)}\le C_Z\|u\|_{L^4(Q_1)}^2
   \end{equation}
   (see Appendix B in Robinson et al.\ (2016) for example).
   To see that the constant is uniform in $\alpha$, given $(\tilde p,\tilde u)$ that satisfy the equations on ${Q_\alpha}$, define $(p,u)$ on $Q_1$ by setting $p(x)=\alpha^2\tilde p(\alpha x)$ and $u(x)=\alpha\tilde u(\alpha x)$. Then
   $$
   [-\Delta p](x)=-\alpha^4(\Delta\tilde p)(\alpha x)\qquad\mbox{and}\qquad \nabla\cdot[(u\cdot\nabla)u](x)=\alpha^4[(\tilde u\cdot\nabla)\tilde u](\alpha x),
   $$
   so $-\Delta p=\nabla\cdot[(u\cdot\nabla)u]$, whence $(p,u)$ satisfy \eqref{CZQ1}. Now observe that  $\|p\|_{L^2(Q_1)}=\alpha^{1/2}\|\tilde p\|_{L^2({Q_\alpha})}$ and $\|u\|_{L^4(Q_1)}=\alpha^{1/4}\|\tilde u\|_{L^2({Q_\alpha})}$ to obtain \eqref{CZ-p}.
\end{proof}

\subsection{Convergence in $L^2(0,T;H^1(\R^3))$ when $u_\alpha\in H^1(Q_\alpha)$}

For initial $u_\alpha^0\in \dot L^2_\sigma(Q_\alpha)\cap H^1(Q_\alpha)$, such that $\tilde u_\alpha^0\to u^0$ in $H^1(\R^3)$, the following theorem shows that the corresponding strong solutions converge in $L^2(0,T;H^1(\R^3))$. One particular example of such a family is provided by Theorem \ref{wa2w}: take a fixed compactly-supported vorticity, and set $u_\alpha^0=\curl_\alpha^{-1}\omega$ and $u^0=\curl^{-1}\omega$. Alternatively, simply take a compactly-supported initial condition $u^0\in H^1_\sigma(\R^3)$ and let $u^0_\alpha=u^0|_{Q_\alpha}$ once $\alpha$ is sufficiently large.

  There is a uniform time for which the existence of a smooth solution $u_\alpha$ (on ${Q_\alpha}$) and $u$ (on $\R^3$) can be guaranteed, starting with this initial condition. The following theorem shows that the extended solutions $\tilde u_\alpha$ must converge to $u$. That there is weak convergence [as in Proposition \ref{prop:conv1-weak}] is fairly standard and follows directly from uniform bounds on $u_\alpha$; that the convergence is strong in $L^2(0,T;H^1(\R^3))$ is more surprising, and requires a more careful analysis. This strong convergence is crucial for the `transference of regularity' result that follows in Section \ref{sec:transfer}.

\begin{theorem}\label{strong}
  Suppose that $u_0\in H^1_\sigma(\R^3)$, $u_\alpha^0\in\dot H^1_\sigma({Q_\alpha})$, and $\tilde u_\alpha^0\to u^0$ in $H^1(\R^3)$, with $\|u_\alpha^0\|_{H^1({Q_\alpha})}^2\le M$ for all $\alpha\ge\alpha_0$.

Set $T=T(M)$ from Theorem \ref{thm:strong}. Denote by $u_\alpha$ the strong solution of the Navier--Stokes equations on ${Q_\alpha}$  with initial data $u_\alpha^0$, and by $u$ the solution on $\R^3$ with initial data $u^0$; all of these solutions exist on $[0,T]$. Then for all $1\le s<2$
  \begin{equation}\label{thepoint}
  \tilde u_\alpha\to u\qquad\mbox{in}\qquad L^r(0,T;H^{1+s}(\R^3)),\qquad r\in[1,2/(s-1));
  \end{equation}
  in particular, $\tilde u_\alpha\to u$ strongly in $L^r(0,T;H^1(\R^3))$ for all $r\in[1,\infty)$.
  \end{theorem}

  \begin{proof}
   Since the solution $u_\alpha$ is smooth on $[0,T]$ it is admissible to take the inner product with $u_\alpha$ in $L^2({Q_\alpha})$ to obtain
  \begin{equation}\label{uaei}
  \frac{1}{2}\|u_\alpha(t)\|_{L^2(Q_\alpha)}^2+\int_0^t\|\nabla u_\alpha(s)\|_{L^2(Q_\alpha)}^2\,\d s\le \frac{1}{2}\|u_0^\alpha\|_{L^2(Q_\alpha)}^2\le \frac{M}{2}.
  \end{equation}
  This gives bounds on $u_\alpha$ in $L^\infty(0,T;L^2({Q_\alpha}))$ and $L^2(0,T;H^1({Q_\alpha}))$ that are uniform with respect to $\alpha$.

     Equation \eqref{uaei} shows that the solutions $u_\alpha$ satisfy the energy inequality \eqref{wk-EI}, so Proposition \ref{prop:conv1-weak} already guarantees that a subsequence (at least) converges to a weak solution on $\R^3$ with initial data $u^0$. However, although $u^0$ gives rise to a strong solution, weak-strong uniqueness (see Theorem 6.10 in Robinson et al.\ (2016), for example) cannot be used here, since the limiting solution $u$ from Proposition \ref{prop:conv1-weak} does not necessarily satisfy the energy inequality (which is required in the proof of weak-strong uniqueness).



Better convergence of $\tilde u_\alpha$ to $u$ can be obtained via bounds on $u_\alpha$ in $H^1$ and bounds on $u_\alpha$ in $H^2$. Take the inner product of the equation with $-\Delta u_\alpha$ in $L^2(Q_\alpha)$ to obtain
  \begin{align*}
  \frac{1}{2}\frac{\d}{\d t}\|\nabla u_\alpha\|_{L^2(Q_\alpha)}^2+\|\Delta u_\alpha\|_{L^2(Q_\alpha)}^2&=\<(u_\alpha\cdot\nabla)u_\alpha,\Delta u_\alpha\>_{L^2(Q_\alpha)}\\
  &\le \|u_\alpha\|_{L^\infty({Q_\alpha})}\|\nabla u_\alpha\|_{L^2({Q_\alpha})}\|\Delta u_\alpha\|_{L^2({Q_\alpha})}\\
  &\le C_A\|\nabla u_\alpha\|_{L^2(Q_\alpha)}^{3/2}\|\Delta u_\alpha\|_{L^2(Q_\alpha)}^{3/2},
  \end{align*}
  where the constant $C_A$ does not depend on $\alpha$ (see Lemma \ref{goodC}). It follows that
  \begin{equation}\label{nablau}
  \frac{\d}{\d t}\|\nabla u_\alpha\|_{L^2(Q_\alpha)}^2+\|\Delta u_\alpha\|_{L^2(Q_\alpha)}^2\le \frac{27}{16}C_A^4\|\nabla u_\alpha\|_{L^2(Q_\alpha)}^6,
  \end{equation}
  and therefore
  \begin{equation}\label{nuat}
  \|\nabla u_\alpha(t)\|^2_{L^2(Q_\alpha)}\le \frac{\|\nabla u_0^\alpha\|_{L^2(Q_\alpha)}^2}{\sqrt{1-\frac{27}{8}c_A^4t\|\nabla u_0^\alpha\|_{L^2(Q_\alpha)}^4}}.
  \end{equation}
  Taking $T=2/[9c_A^4M^2]$ it follows that
  $$
  \|\nabla u_\alpha(t)\|_{L^2(Q_\alpha)}^2\le 2\|\nabla u_0^\alpha\|_{L^2(Q_\alpha)}^2\le 2M\qquad\mbox{for all }t\in[0,T]
  $$
  and, integrating \eqref{nablau} from $0$ to $T$ and using the bound in \eqref{nuat}, that
  \begin{equation}\label{forH2bound}
  \int_0^T\|\Delta u_\alpha(t)\|_{L^2(Q_\alpha)}^2\,\d t\le \frac{5M}{2}.
  \end{equation}
  Therefore $u_\alpha$ is bounded uniformly in $L^\infty(0,T;H^1({Q_\alpha}))$ and in $L^2(0,T;H^2({Q_\alpha}))$.

 To obtain bounds on the time derivative, since the equation
  $$
  \partial_tu_\alpha=\Delta u_\alpha-(u_\alpha\cdot\nabla)u_\alpha-\nabla p_\alpha
  $$
  holds as an equality in $L^2(0,T;L^2(Q_\alpha))$ it follows that
  $$
  \|\partial_tu_\alpha\|_{L^2(Q_\alpha)}\le\|\Delta u_\alpha\|_{L^2(Q_\alpha)}+\|(u_\alpha\cdot\nabla)u_\alpha\|_{L^2(Q_\alpha)}+\|\nabla p_\alpha\|_{L^2(Q_\alpha)}.
  $$
  The Helmholtz decomposition provides a bound on $\nabla p_\alpha$ in $L^2(Q_\alpha)$: write
  $$
  L^2(Q_\alpha)=L^2_\sigma(Q_\alpha)\oplus G(Q_\alpha),
  $$
  where
  $$
  G(Q_\alpha)=\{\nabla\psi:\ \psi\in H^1(Q_\alpha)\}.
  $$
  These two spaces are orthogonal: for any $v\in H(Q_\alpha)$ and $\nabla\psi\in G(Q_\alpha)$
  $$
  \<v,\nabla\psi\>_{L^2(Q_\alpha)}=0.
  $$

Take any $\phi\in L^2({Q_\alpha})$ and write $\phi=v+\nabla\psi$, where $v\in H(Q_\alpha)$ and $\nabla\psi\in G(Q_\alpha)$. Then
  $$
  \<\nabla p_\alpha,\phi\>=\<\nabla p_\alpha,\nabla\psi\>=\<\partial_tu_\alpha-\Delta u_\alpha+(u_\alpha\cdot\nabla)u_\alpha,\nabla\psi\>=\<(u_\alpha\cdot\nabla)u_\alpha,\nabla\psi\>,
  $$
  since $\nabla\psi$ is orthogonal to any divergence-free function. It follows that
  \begin{align*}
  |\<\nabla p_\alpha,\phi\>|&\le\|(u_\alpha\cdot\nabla)u_\alpha\|_{L^2({Q_\alpha})}\|\nabla\psi\|_{L^2({Q_\alpha})}\\
  &\le\|(u_\alpha\cdot\nabla)u_\alpha\|_{L^2({Q_\alpha})}\|\phi\|_{L^2({Q_\alpha})},
  \end{align*}
  which shows that
  $$
  \|\nabla p_\alpha\|_{L^2({Q_\alpha})}\le\|(u_\alpha\cdot\nabla)u_\alpha\|_{L^2({Q_\alpha})}.
  $$
  It follows that
  \begin{align*}
  \|\partial_tu_\alpha\|_{L^2({Q_\alpha})}&\le \|\Delta u_\alpha\|_{L^2({Q_\alpha})}+2\|(u_\alpha\cdot\nabla)u_\alpha\|_{L^2({Q_\alpha})}\\
  &\le   \|\Delta u_\alpha\|_{L^2({Q_\alpha})}+2\|u_\alpha\|_{L^\infty({Q_\alpha})}\|\nabla u_\alpha\|_{L^2({Q_\alpha})},
 \end{align*}
  so $\partial_tu_\alpha$ is bounded uniformly in $L^2(0,T;L^2({Q_\alpha}))$.

  All these bounds carry over uniformly to the extended functions $\tilde u_\alpha$, which are therefore bounded uniformly in $L^\infty(0,T;L^2(\R^3))$ and $L^2(0,T;H^1(\R^3))$, with $\partial_t\tilde u_\alpha$ bounded uniformly in $L^2(0,T;L^2(\R^3))$.

  It follows -- using weak-$*$ sequential compactness, weak sequential compactness in reflexive Banach spaces (see Chapter 27 in Robinson, 2020, for example), and the Aubin--Lion compactness theorem (see Simon, 1987), that there is a subsequence $\tilde u_{\alpha_j}$ that converges to some limit $u\in L^\infty(0,T;H^1(\R^3))\cap L^2(0,T;H^2(\R^3))$, with
  $$
  \tilde u_{\alpha_j}\cgws u\mbox{ in }L^\infty(0,T;H^1(\R^3)),\qquad \tilde u_{\alpha_j}\cgw u\mbox{ in }L^2(0,T;H^2(\R^3)),
  $$
  and $\tilde u_{\alpha_j}\to u$ strongly in $L^2(0,T;H^1(K))$ for every compact subset $K$ of $\R^3$.

We know from before that $u$ is at least a weak solution on $[0,T)$: these bounds now show that $u$ has the required regularity to be a strong solution. By the uniqueness of strong solutions (in their own class) it follows that in fact $\tilde u_\alpha$ converges to $u$ in all senses above as $\alpha\to\infty$, and not only through the sequence $\alpha_j$. (See Lemma 3.1 in Robinson (2004), for example.)

  %
%
%
%
%
%
%
%

  To obtain strong convergence of $\tilde u_\alpha$ to $u$ solutions in $L^2(0,T;H^1(\R^3))$, the idea is first to use Lemma \ref{Leray-idea} to prove that $\tilde u_\alpha\to u$ in $L^p(0,T;L^2(\R^3))$, $p\in[1,\infty)$, by showing that
  \begin{equation}\label{whatwewant}
 \int_{x\in {Q_\alpha}:\ |x|\ge R}|u_\alpha(t)|^2
  \end{equation}
  can be made small (uniformly for $\alpha$ sufficiently large and $t\in[0,T]$) by taking $R$ large. Towards this, observe that it follows from the assumptions on $u_\alpha^0$ that for every $\eta>0$ there exists $r=r(\eta)$ and $\beta=\beta(\eta)\ge r(\eta)$ such that
   \begin{equation}\label{4Leray}
   \int_{x\in {Q_\alpha}:\ |x|\ge r}|u^0_\alpha(x)|^2\,\d x<\eta\qquad\mbox{for every }\alpha\ge\beta.
   \end{equation}

  To obtain the bound \eqref{whatwewant} on $u_\alpha$, take the inner product [in $L^2({Q_\alpha})$] of
   $$
   \partial_tu_\alpha-\Delta u_\alpha+(u_\alpha\cdot\nabla)u_\alpha+\nabla p_\alpha=0
   $$
 with $\varrho_\alpha u_\alpha$, where $\varrho_\alpha$ is the function defined in \eqref{eq:varrho}.

  Then (cf.\ proof of Proposition 14.3 in Robinson et al., 2016) an integration by parts yields
  \begin{align*}
  \frac{1}{2}&\frac{\d}{\d t}\int_{{Q_\alpha}}\varrho_\alpha|u_\alpha|^2+\int_{{Q_\alpha}}\varrho_\alpha|\nabla u_\alpha|^2\\
  &=-\int_{{Q_\alpha}}(\partial_ju_{\alpha,i})u_{\alpha,i}(\partial_j\varrho_\alpha)+\int_{{Q_\alpha}}|u_\alpha|^2(u_\alpha\cdot\nabla)\varrho_\alpha
  +\int_{{Q_\alpha}} p_\alpha(u_\alpha\cdot\nabla)\varrho_\alpha.
    \end{align*}
  Integrating from $0$ to $t$ and using the definition of $\varrho_\alpha$ yields
  \begin{align*}
  \frac{1}{2}\int_{x\in {Q_\alpha}:\ |x|>R}&|u_\alpha(t)|^2\le\frac{1}{2}\int_{x\in {Q_\alpha}:\ |x|>r}|u_\alpha^0|^2\\
  &+\frac{1}{R-r}\int_0^t\int_{{Q_\alpha}} |\nabla u_\alpha||u_\alpha|+|u_\alpha|^3+|p_\alpha||u_\alpha|.
  \end{align*}
  Since $\|u_\alpha(s)\|_{L^2({Q_\alpha})}\le\|u_\alpha^0\|_{L^2({Q_\alpha})}$ the second term on the right-hand side can be bounded by
  $$
  \frac{1}{R-r}\,\|u_\alpha^0\|_{L^2({Q_\alpha})}\int_0^t\|\nabla u_\alpha(s)\|_{L^2({Q_\alpha})}+\|u_\alpha(s)\|_{L^4({Q_\alpha})}^2+\|p_\alpha\|_{L^2({Q_\alpha})}\,\d s.
  $$
 The first term of this integral can be estimated by
   $$
   \int_0^t\|\nabla u_\alpha(s)\|_{L^2({Q_\alpha})}\,\d s\le t^{1/2}\int_0^t\|\nabla u_\alpha(s)\|_{L^2({Q_\alpha})}^2\,\d s.
   $$
    Using the Calder\'on--Zygmund estimate $\|p_\alpha\|_{L^2({Q_\alpha})}\le C_Z\|u_\alpha\|_{L^4({Q_\alpha})}^2$ from \eqref{CZ-p} the second and third terms can be combined; then using the Lebesgue interpolation inequality $\|f\|_{L^4}\le\|f\|_{L^2}^{1/4}\|f\|_{L^6}^{3/4}$ and the Sobolev embedding $\|f\|_{L^6({Q_\alpha})}\le C_6\|\nabla f\|_{L^2({Q_\alpha})}$ from \eqref{L6ineq}
  \begin{align*}
  \int_0^t\|u_\alpha(s)\|_{L^4({Q_\alpha})}^2\,\d s&\le\int_0^t\|u_\alpha(s)\|_{L^2({Q_\alpha})}^{1/2}\
  |u_\alpha(s)\|_{L^6({Q_\alpha})}^{3/2}\,\d s\\
  &\le C_6^{3/2}\|u_\alpha^0\|_{L^2({Q_\alpha})}^{1/2}\int_0^t\|\nabla u_\alpha(s)\|_{L^2({Q_\alpha})}^{3/2}\,\d s\\
  &\le C_6^{3/2}\|u_\alpha^0\|_{L^2({Q_\alpha})}^{1/2}t^{1/4}\left(\int_0^t\|\nabla u_\alpha(s)\|_{L^2({Q_\alpha})}^2\,\d s\right)^{3/4}.
  \end{align*}

  Therefore, for all $t\in[0,T]$,
\begin{align*}
 \frac{1}{2}\int_{x\in {Q_\alpha}:\ |x|>R}&|u_\alpha(t)|^2\le\frac{1}{2}\int_{x\in {Q_\alpha}:\ |x|>r}|u_\alpha^0|^2\\
&+\frac{\|u_\alpha^0\|_{L^2(Q_\alpha)}}{R-r}\bigg[T^{1/2}\int_0^T\|\nabla u_\alpha(s)\|_{L^2({Q_\alpha})}^2\,\d s\\
&\qquad+2C_6^{3/2}\|u_\alpha^0\|_{L^2({Q_\alpha})}^{1/2}T^{1/4}\left(\int_0^T\|\nabla u_\alpha(s)\|_{L^2({Q_\alpha})}^2\,\d s\right)^{3/4}\bigg];
 \end{align*}
 or
 $$
  \int_{x\in {Q_\alpha}:\ |x|>R}|u_\alpha(t)|^2\le\int_{x\in {Q_\alpha}:\ |x|>r}|u_\alpha^0|^2+\frac{\Gamma}{R-r},
  $$
  where $\Gamma$ can be chosen to be independent of $\alpha$. Given $\eta>0$, it follows from \eqref{4Leray} that there exist $\beta$ and $r$ such that
  $$
  \int_{x\in {Q_\alpha}:\ |x|>r}|u_\alpha^0|^2<\eta/2\qquad\mbox{for }\alpha\ge\beta.
  $$
  Now choose $R$ sufficiently large that $\Gamma/(R-r)<\eta/2$, and then increase $\beta$ if necessary so that $\beta>R+1$. There therefore exist $R(\eta)$ and $\beta(\eta)$ such that
  $$
  \int_{x\in {Q_\alpha}:\ |x|>R(\eta)}|u_\alpha(t)|^2\le\eta\qquad\mbox{for }\alpha\ge\beta(\eta),\ t\in[0,T],
  $$
  with $\beta(\eta)>R(\eta)+1$, which was \eqref{whatwewant}. Finally, it follows from \eqref{mest} that
  \begin{equation}\label{thanks}
 \int_{|x|>R(\eta)}|\tilde u_\alpha(t)|^2\le 27\eta\qquad\mbox{for }\alpha\ge\beta(\eta).
  \end{equation}

 Since $\tilde u_\alpha\to u$ in $L^2(0,T;L^2(K))$ for every compact subset $K$ of $\R^3$, it follows that $\tilde u_\alpha(t)\to u(t)$ in $L^2(B(0,n))$ for every $n\in\N$ and for almost every $t\in\R$. Given the estimate in \eqref{thanks}, it now follows from Lemma \ref{Leray-idea} that $\tilde u_\alpha(t)\to u(t)$ in $L^2(\R^3)$ for almost every $t$, i.e.\ $\|\tilde u_\alpha(t)-u(t)\|_{L^2(\R^3)}\to0$ for almost every $t$. Now observe that
 \begin{align*}
 \|\tilde u_\alpha(t)-u(t)\|_{L^2(\R^3)}&\le\|\tilde u_\alpha(t)\|_{L^2(\R^3)}+\|u(t)\|_{L^2(\R^3)}\\
 &\le 27\|u_\alpha(t)\|_{L^2(Q_\alpha)}+\|u^0\|_{L^2(\R^3)}\\
 &\le 27\|u_\alpha^0\|_{L^2(Q_\alpha)}+\|u^0\|_{L^2(\R^3)}\le 28\sqrt M;
 \end{align*}
 it follows, using the Dominated Convergence Theorem, that $\tilde u_\alpha\to u$ in $L^2(0,T;L^2(\R^3))$ (and in fact for every $p\in[1,\infty)$).

The fact that $\tilde u_\alpha\to u$ strongly in $L^2(0,T;L^2(\R^3))$ can now be used to improve the convergence of $\tilde u_\alpha$ to $u$ from weak in $L^2(0,T;H^1(\R^3))$ to strong in $L^r(0,T;H^1(\R^3))$ for all $r\in[1,\infty)$; rather than having to bound the `tails' of $\int_{|x|\ge R} |\nabla \tilde u_\alpha|^2$, all that is required is the additional information that $\tilde u_\alpha$ is uniformly bounded in $L^\infty(0,T;H^1(\R^3))$ and in $L^2(0,T;H^2(\R^3))$ (which is guaranteed by \eqref{forH2bound}). Assume that $r\ge 2$; given convergence in any such $L^r(0,T;H^1(\R^3))$, convergence with $r\in[1,2)$ follows immediately. Now note that the Sobolev interpolation inequality
$$
\|f\|_{H^1(\R^3)}\le C\|f\|_{L^2(\R^3)}^{1/2}\|f\|_{H^2(\R^3)}^{1/2}
$$
implies that
\begin{align*}
\int_0^T\|\tilde u_\alpha&-u\|_{H^1(\R^3)}^r\,\d t\le \|\tilde u_\alpha-u\|_{L^\infty(0,T;H^1(\R^3))}^{r-2}\int_0^T\|\tilde u_\alpha-u\|_{H^1(\R^3)}^2\,\d t\\
&\le C\|\tilde u_\alpha-u\|_{L^\infty(0,T;H^1(\R^3))}^{r-1}\int_0^T\|\tilde u_\alpha-u\|_{L^2(\R^3)}\|\tilde u_\alpha-u\|_{H^2(\R^3)}\,\d t\\
&\le C\|\tilde u_\alpha-u\|_{L^\infty(0,T;H^1(\R^3))}^{r-1}\left(\int_0^T\|\tilde u_\alpha-u\|_{L^2(\R^3)}^2\,\d t\right)^{1/2}\\
&\qquad\qquad\qquad\qquad\qquad\qquad\qquad\times\left(\int_0^T\|\tilde u_\alpha-u\|_{H^2(\R^3)}^2\,\d t\right)^{1/2}.
\end{align*}
%
%
%
Since $\tilde u_\alpha$ (and hence $u$) are uniformly bounded in $L^\infty(0,T;H^1(\R^3))$ and in  $L^2(0,T;H^2(\R^3))$, this implies that $\tilde u_\alpha\to u$ in $L^r(0,T;H^1(\R^3))$ as claimed.

To finish the proof, if $s=1+\theta$ with $\theta\in(0,1)$ and $r\in[1,2/\theta)$, then
$$
\|f\|_{H^{1+\theta}(\R^3)}\le C\|f\|_{H^1(\R^3)}^{1-\theta}\|f\|_{H^2(\R^3)}^\theta
$$
and so
\begin{align*}
\int_0^T\|\tilde u_\alpha&-u\|_{H^{1+\theta}(\R^3)}^r\,\d t\le C\int_0^T\|\tilde u_\alpha-u\|_{H^1(\R^3)}^{(1-\theta)r}\|\tilde u_\alpha-u\|_{H^2(\R^3)}^{\theta r}\,\d t\\
&\le C\left(\int_0^T\|\tilde u_\alpha-u\|_{H^1(\R^3)}^{2r(1-\theta)/(2-r\theta)}\,\d t\right)^{(2-r\theta)/2}\\
&\qquad\qquad\qquad\qquad\qquad\times\left(\int_0^T\|\tilde u_\alpha-u\|_{H^2(\R^3)}^2\,\d t\right)^{r\theta/2}.\qedhere
\end{align*}
\end{proof}



\section{`Transfer of regularity' from the whole space to the periodic case}\label{sec:transfer}

This final section shows that the existence of a solution on the whole space for a particular choice of initial condition is transferred to the periodic case when $\alpha$ is large enough.

\subsection{The transfer of regularity result}

The following theorem shows that if $u_0$ gives rise to a smooth solution on $[0,T^*]$ on the whole space, the corresponding periodic problems will have smooth solutions on the same time interval once the size of the periodic domain is sufficiently large. Note that $T^*$ does not need to be a `guaranteed local existence time' from the proof of the existence of strong solutions, but could be significantly longer.

The simplest particular cases of the theorem are when $u^0_\alpha\equiv u^0\in \dot H^1_\sigma(\R^3)$ for all $\alpha$ sufficiently large or when $u^0_\alpha=\curl_\alpha^{-1}\omega_0$ for some $\omega_0\in \dot H^1_\sigma(\R^3)$.
\begin{theorem}\label{thm:transfer}
  Suppose that $u_\alpha^0\in\dot H^1_\sigma({Q_\alpha})$ and $u_0\in H^1_\sigma(\R^3)$, with $\tilde u_\alpha^0\to u^0$ in $H^1(\R^3)$. Suppose in addition that there exists $T^*>0$ such that the equations on $\R^3$ with initial condition $u^0$ admit a solution
  $$
  u\in L^\infty([0,T^*];H^1(\R^3))\cap L^2(0,T^*;H^2(\R^3)).
  $$
  Then for $\alpha$ sufficiently large the equations on the periodic domain ${Q_\alpha}$ with initial data $u_\alpha^0$ have a smooth solution
  $$
  u_\alpha\in L^\infty(0,T^*;H^1({Q_\alpha}))\cap L^2(0,T^*;H^2({Q_\alpha}))
  $$
  and $\tilde u_\alpha\to u$ in $L^r(0,T^*;H^1)$, $r\in[1,\infty)$, as $\alpha\to\infty$.
  \end{theorem}

  \begin{proof} Since $u\in L^\infty([0,T^*];H^1(\R^3))$ there exists $M>0$ such that
     $$
     \|u(t)\|_{H^1(\R^3)}^2\le M\qquad\mbox{for all }t\in[0,T^*].
     $$
     Theorem \ref{thm:strong} guarantees that there exists a uniform time $\tau$ such that any solution with $u(0)=v_0$, where $\|v_0\|_{H^1(\R^3)}^2\le 2M$, exists at least on the time interval $[0,\tau]$.

     Set $N=2T^*/\tau$ and fix $r\in[1,\infty)$.

     Theorem \ref{strong} ensures that $\tilde u_\alpha\to u$ in $L^r(0,T;H^1(\R^3))$ as $\alpha\to\infty$. In particular, $\tilde u_\alpha(t)\to u(t)$ in $H^1(\R^3)$ for almost every $t\in(0,\tau)$; choose one such $t$ with $t>\tau/2$ and call this $t_1$.

Choose $\alpha_1$ such that $\|\tilde u_\alpha(t_1)\|_{H^1(\R^3)}\le 2M$ for all $\alpha\ge \alpha_1$. Since
$$
\|u_\alpha(t_1)\|_{H^1({Q_\alpha})}\le\|\tilde u_\alpha(t_1)\|_{H^1(\R^3)},
$$
this bound is enough to ensure that, uniformly for $\alpha\ge\alpha_1$, the solutions on ${Q_\alpha}$ starting from $u_\alpha(t_1)$ exist on the time interval $[t_1,t_1+\tau]\supset[\tau,3\tau/2]$.

 Since $\tilde u_\alpha(t_1)\to u(t_1)$ in $H^1(\R^3)$, Theorem \ref{strong} can again be used to guarantee that as $\alpha\to\infty$ ($\alpha\ge\alpha_1$), have $\tilde u_\alpha\to u$ in $L^r(t_1,t_1+\tau;H^1(\R^3))$. Again, the convergence in $H^1(\R^3)$ for almost-every time means that there exists $t_2\in(t_1,t_1+\tau)$ with $t_2>t_1+\tau/2>\tau$ such that $\tilde u_\alpha(t_2)\to u(t_2)$ in $H^1(\R^3)$; in particular, there exists $\alpha_2\ge\alpha_1$ such that $\|u_\alpha(t_2)\|_{H^1({Q_\alpha})}\le 2M$ for all $\alpha\ge\alpha_2$.

  Continue in this way, noting that at each step the interval of existence of the solutions on ${Q_\alpha}$ (for $\alpha\ge\alpha_n$) increases by at least $\tau/2$. After $N$ steps the entire interval $[0,T^*]$ has been covered, showing that the solution on ${Q_\alpha}$ starting at $u_\alpha^0$ is strong on $[0,T^*]$ for all $\alpha\ge\alpha_N$.\end{proof}

Note that this result does not say that if the equations are regular on $\R^3$ -- i.e.\ if \emph{any} smooth (compactly-supported) initial condition gives rise to a smooth solution for all $t>0$ -- then they are regular on $Q_\alpha$ for $\alpha$ large enough (which would then imply regularity on $Q_\alpha$ for any $\alpha$). Rather, for a fixed (compactly-supported) initial condition, regularity on $\R^3$ on a given time interval carries over to $Q_\alpha$ for $\alpha$ sufficiently large.

A full `transfer of regularity' from one problem to another would require a convergence result in which the distance between solutions on $\R^3$ and $Q_\alpha$ could be bounded in terms of the $H^1$ norm of the initial data, which appears to require much more sophisticated methods that the compactness-based arguments employed here. [For results in this direction for the Ginzburg--Landau equation see Mielke (1997) and for the two-dimensional Navier--Stokes equations see Zelik (2013).]

\section*{Conclusion}

Given fixed sufficiently regular initial data with compact support, solutions of the Navier--Stokes equations on expanding periodic domains converge to the corresponding solution on the whole space; and this can to some extent be `reversed', in that a compactly-supported initial condition that leads to a strong solution on a time interval $[0,T^*]$ (which could be significantly longer than what is guaranteed by standard existence theorems) will give rise to a strong solution on the same time interval on a sufficiently large periodic domain.

It is natural to conjecture that a similar result holds given any choice of smooth, simply-connected, bounded subset $\Omega$ of $\R^3$, replacing $(-\alpha,\alpha)^3$ by $\alpha\Omega$ and imposing no-slip (Dirichlet) boundary conditions on the boundary of $\alpha\Omega$. However, the estimates on the pressure required in the proof given here become much more delicate in the case of a bounded domain (see Sohr \& von Wahl, 1986, for example).

While the results here demonstrate convergence, they give no error estimates; this appears to be a significantly harder problem, but a particularly interesting one if one is to view solving the equations on a periodic domain as a `numerical approximation' to the solution of the equations on the whole space.

\section*{Acknowledgments}

Many thanks to Robert Kerr, for interesting discussions that motivated the problem considered here. Thanks also to Wojciech O\.za\'nski and Jos\'e Rodrigo for helpful conversations about some of the issues that needed to be resolved, and for numerous insightful comments on a preliminary version of the paper. In particular, Wojciech pointed out that a minor change in the argument would give convergence in $L^p(0,T;H^1(\R^3))$ for all $1\le p<\infty$.

     \section*{References}\parindent0pt\parskip7pt

     Ambrose, D.M., Kelliher, J.P., Lopes Filho, M.C., Nussenzveig Lopes, H.J. (2015) Serfati solutions to the 2D Euler equations on exterior domains. \textit{J. Diff. Eq.} {\bf 259}, 4509--4560.

     Constantin, P. (1986) Note on loss of regularity for solutions of the 3-D incompressible Euler and related equations. \textit{Commun. Math. Phys.} {\bf 104}, 311--326.

     Chernyshenko, S.I., Constantin, P., Robinson, J.C., \& Titi, E.S. (2007) A posteriori regularity of the three-dimensional Navier--Stokes equations from numerical computations. \textit{J. Math. Phys.} {\bf 48}, 065204.

     Doering, C.R., \& Gibbon, J.D. (1995) \textit{Applied analysis of the Navier--Stokes equations}. Cambridge University Press, Cambridge, UK.

      Enciso, A., García-Ferrero, M.A., \& Peralta-Salas, D. (2018) The Biot--Savart operator of a bounded domain. \textit{J. Math. Pures Appl.} {\bf 119}, 85--113.

     Evans, L.C. (2010) \textit{Partial differential equations}. American Mathematical Society, Providence, RI.

     Gallagher, I. (1997) The tridimensional Navier--Stokes equations with almost bidimensional data: stability, uniqueness, and life span. \textit{Internat. Math. Res. Notices} {\bf 18}, 919--935.


 Heywood, J.G. (1976) On uniqueness questions in the theory of viscous flow. \textit{Acta Math.} {\bf 136}, 61--102.

     Heywood, J.G. (1988) Epochs of regularity for weak solutions of the Navier--Stokes equations in unbounded domains. \textit{Tohoku Math} {\bf 40}, 293--313.

     Hopf, E. (1951) \"Uber die Aufgangswertaufgave f\"ur die hydrodynamischen Grundliechungen. \textit{Math. Nachr.} {\bf 4}, 213--231.

     Kato, T. (1972) Nonstationary flow of viscous and ideal fluids in $\R^3$. \textit{J. Funct. Anal.} {\bf 9}, 296--305.

     Kerr, R.M. (2018) Enstrophy and circulation scaling for Navier--Stokes reconnection. \textit{J. Fluid Mech.} {\bf 839}, R2.

     Leray, J. (1934) Essai sur le mouvement d'un liquide visqueux emplissany l'espace. \textit{Acta Math.} {\bf 63}, 193--248.

     Ladyzhenskaya, O.A. (1969) \textit{The mathematical theory of viscous incompressible flow}. Gordon and Breach, New York, NY.

     Mielke, A. (1997) The complex Ginzburg--Landau equation on large and unbounded domains: sharper bounds and attractors. \textit{Nonlinearity} {\bf 10}, 199--222.

     O\.za\'nski, W. \& Pooley, B. (2018) Leray's fundamental work on the Navier--Stokes equations: a modern review of ``Sur le mouvement d'un liquid visqueux emplissant l'espace". pp.~113--203 in Fefferman, C.L., Rodrigo, J.L., \& Robinson, J.C. (Eds.) \emph{Partial differential equations in fluid mechanics}. LMS Lecture Notes. Cambridge University Press, Cambridge, UK.

     Raugel, G. \& Sell, G.R. (1993) Navier--Stokes equations on thin 3D domains. I. Global attractors and global regularity of solutions. \textit{J. Amer. Math. Soc.} {\bf 6}, 503--568.

     Robinson, J.C. (2004) A coupled particle-continuum model: well-posedness and the limit of zero radius. \textit{Roy. Soc. London Proc. A} {\bf 460}, 131--1334.

     Robinson, J.C. (2020) \textit{An introduction to functional analysis}. Cambridge University Press, Cambridge, UK.

     Robinson, J.C., Rodrigo, J.L., \& Sadowski, W. (2016) \textit{The three-dimensional Navier--Stokes equations}. Cambridge studies in advanced mathematics 157. Cambridge University Press, Cambridge, UK.

     Scheeler, M.W., Kleckner, D., Proment, D., Kindlmann, G.L., \& Irvine, W.T.M. (2014) Helicity conservation by flow across scales in reconnecting vortex links and knots. \textit{Proc. Natl Acad. Sci. USA} {\bf 111}, 15350--15355.

     Serfati, P. (1995) Solutions $C^\infty$ en temps, $n$-log Lipschitz born\'ees en espace et \'equation d'Euler. \textit{C. R. Acad. Sci. Paris S\'er. I Math.} {\bf 320}, 555--558.

     Simon, J. (1987) Compact sets in the space $L^p(0,T;B)$. \textit{Ann. Mat. Pura Appl.} {\bf 146}, 65--96.

     Sohr, H. \& von Wahl, W. (1986) On the regularity of the pressure of weak solutions of Navier--Stokes equations.     \emph{Arch. Math. (Basel)} {\bf 46}, 428--439.

     Swann, H. (1971) The convergence with vanishing viscosity of nonstationary Navier--Stokes flow to ideal flow in $\R^3$. \textit{Trans. Amer. Math. Soc.} {\bf 157}, 373--397.

     Tao, T. (2013) Localisation and compactness properties of the Navier--Stokes global regularity problem. \textit{Analysis \& PDE} {\bf 6}, 25--107.

     Zelik, S. (2013) Infinite energy solutions for damped Navier--Stokes equations in $\R^2$. \textit{J. Math. Fluid Mech.} {\bf 15}, 717--745.

\end{document}